\documentclass[12pt]{amsart}
\usepackage{hyperref}
\usepackage{ifdraft}
\usepackage{amssymb,amsmath,amscd,epsfig,amsthm,verbatim,eqnarray}
\usepackage{fullpage,verbatim} 

\newcommand{\Z}{\mathbb Z}
\newcommand{\f}{(f|_v)}
\newcommand{\gz}{(g|_0)}
\newcommand{\gu}{(g|_1)}
\newcommand{\ggz}{g|_0}
\newcommand{\ggu}{g|_1}
\newcommand{\ZZ}{\Z/(2^n\Z)}
\newcommand{\Zd}{\Z/(d^n\Z)}
\newcommand{\Aut}{\mathop{\rm Aut}\nolimits}

\newtheorem{theorem}{Theorem}[section]

\newtheorem{proposition}[theorem]{Proposition}
\newtheorem{lemma}[theorem]{Lemma}

\theoremstyle{definition}
\newtheorem{definition}{Definition}
\newtheorem{example}{Example}
\newtheorem{remark}[theorem]{Remark}

\newtheorem*{theorem_permutational}{Theorem~\ref{the:riv}}
\newtheorem*{theorem_sections}{Theorem~\ref{the:sec}}
\newtheorem*{theorem_main}{Theorem~\ref{the:main}}
\newtheorem*{theorem_larin}{Theorem~\ref{the:larin}}

\title{Endomorphisms of regular rooted trees induced by the action of polynomials on the ring $\Z_d$ of $d$-adic integers}
\author{Elsayed Ahmed}
\address{Department of Mathematics and Statistics\\
University of South Florida\\
4202 E Fowler Ave\\
Tampa, FL 33620-5700\\
\href{mailto:eahmed1@mail.usf.edu}{eahmed1@mail.usf.edu}}

\author{Dmytro Savchuk}
\address{Department of Mathematics and Statistics\\
University of South Florida\\
4202 E Fowler Ave\\
Tampa, FL 33620-5700\\
\href{mailto:savchuk@usf.edu}{savchuk@usf.edu}}

\begin{document}
\maketitle

\begin{abstract}
We show that every polynomial in $\Z[x]$ defines an endomorphism of the $d$-ary rooted tree induced by its action on the ring $\Z_d$ of $d$-adic integers. The sections of this endomorphism also turn out to be induced by polynomials in $\Z[x]$ of the same degree. In the case of permutational polynomials acting on $\Z_d$ by bijections the induced endomorphisms are automorphisms of the tree. In the case of $\Z_2$ such polynomials were completely characterized by Rivest in~\cite{riv:perm}. As our main application we utilize the result of Rivest to derive the condition on the coefficients of a permutational polynomial $f(x)\in\Z[x]$ that is necessary and sufficient for $f$ to induce a level transitive automorphism of the binary tree, which is equivalent to the ergodicity of the action of $f(x)$ on $\Z_2$ with respect to the normalized Haar measure.
\end{abstract}

\section{Introduction}

For a fixed integer $d\geq 2$ every polynomial $f(x) \in \Z[x]$ naturally induces mappings $f_n \colon \Zd\rightarrow\Zd$ for all positive integers $n$. Equivalently, these mappings are induced by the action of $f$ on the ring of $d$-adic integers $\Z_d$. These two equivalent approaches to study polynomials in $\Z[x]$ have been used in different contexts during the last several decades. One of the first motivations came from the constructions of the generators for pseudo-random sequences and goes back to Knuth~\cite{knuth:art2,anashin:uniformly_distributed_seqs98,larin:tr}. In this and most of the other applications it is crucial to consider polynomials acting by permutations on $\Zd$. Such polynomials are generally called permutational polynomials, however it is important to emphasize the distinction of these polynomials from the class of permutation polynomials that permute elements of finite fields $\mathbb F_{p^n}$ (see~\cite[Chapter 7]{lidl_n:finite_fields83} for a survey). In many cases the stronger condition of transitivity of this action is required. Another type of applications come from cryptography. Rivest in~\cite{riv:perm} completely characterized polynomials that act by permutations on $\ZZ$ for all $n\geq 1$ and points out the use of one of them (namely, $f(x)=2x^2+x$) in the symmetric block cipher RC6~\cite{rivest:RC6} that was one of the five finalists of the AES competition. The questions of ergodicity of the action of permutational polynomials have been studied in the context of dynamical systems on $\Z_p$ by Anashin~\cite{ana:erg}. We refer the reader to a nice survey paper~\cite{fan:dynamics_survey} for a background and history in polynomial dynamics on $\Zd$.

In this paper we offer another view to the polynomials acting on $\Z_d$ and on $\Z/(d^n\Z)$. Namely, we use the tools from the theory of groups acting on rooted trees by automorphisms and groups generated by Mealy automata. This theory exploded in 1980's after many counterexamples to well known conjectures in group theory were found among such groups. For example, Grigorchuk group was the first example of a group  of intermediate growth~\cite{grigorch:milnor}, as well as one of the first examples of infinite finitely generated torsion groups~\cite{grigorch:burnside} (see also~\cite{aleshin:burnside,sushch:burnside,gupta_s:burnside}). Now this is a rich theory with connections to combinatorics~\cite{grigorchuk-s:hanoi-cr}, analysis~\cite{grigorch_lsz:atiyah}, holomorphic dynamics~\cite{nekrash:self-similar}, dynamical systems~\cite{grigorch_s:ergodic_decomposition}, computer science~\cite{miasnikov_s:automatic_graph}, and many other areas. We refer the reader to the survey article~\cite{gns00:automata} for history and references. The key idea in many of the arguments in this theory is understanding automorphisms of rooted trees by describing their \emph{sections} (terms \emph{states} and \emph{restrictions} are also widely used) at the subtrees hanging down from vertices of the original rooted tree. If the original tree was regular (i.e., every vertex has the same number of children), these subtrees are canonically isomorphic to the original tree, and these sections can be treated as the automorphisms of the original tree as well. We utilize this approach to analyze the action of polynomials on $\Z_d$.

Note that the connection between the functions on the boundary of $X^*$ induced by automata and 1-Lipschitz functions on $\Z_d$ was also established by Anashin in~\cite{anashin:automata12}, where a criterion for finiteness of the corresponding automaton in terms of the van der Put series of the function was developed. This criterion provided an application of the $p$-adic analysis to the theory of automata. In this paper we suggest a converse application.

The set of vertices of a rooted $d$-ary tree can be identified with the set $X^*$ of all finite words over an alphabet $X=\{0,1,\ldots,d-1\}$. In this case the $n$-th level of $X^*$ corresponds to $\Zd$ by identifying $x_0x_1 \ldots x_{n-1} \in X^n$ with $x_0+d\cdot x_1+ \cdots +d^{n-1}\cdot x_{n-1} \in \Zd$. The boundary of $X^*$ (consisting of all infinite paths without backtracking initiating at the root) is then identified with the ring $\Z_d$ of $d$-adic integers. We show in Proposition~\ref{pro:distinct} that in this interpretation each polynomial in $\Z[x]$ induces an endomorphism of the tree $X^*$, while each permutational polynomial induces an automorphism. Our first result describes the structure of these endomorphisms.

\begin{theorem_sections}
Given a polynomial $f(x)=a_0+a_1x+ \cdots +a_tx^t \in \Z[x]$ inducing an endomorphism of $X^*$, the image of a vertex $x_0 \in X$ under $f$ is $f(x_0)\pmod d$ and the section of $f$ at $x_0$ is again a polynomial given by the equation:
\begin{equation}
\label{equ:secf}
f|_{x_0}(x)=\bigl(f(x_0)\ \mathrm{div}\  d\bigr)+\sum_{i=1}^{t}\frac{f^{(i)}(x_0)}{i!}d^{i-1}x^i,
\end{equation}
where $f^{(i)}$ denotes the $i^{th}$ derivative of a polynomial $f$, and $f(x_0)\ \mathrm{div}\  d$ is the quotient of $f(x_0)$ and $d$.
\end{theorem_sections}

Note that the case of linear polynomials was partially considered by Bartholdi and \v Sun\'ic in~\cite{bartholdi_s:bsolitar}.

Our main application deals with permutational polynomials acting transitively on $\ZZ$ for all $n$. In terms of the action on the tree this condition is equivalent to being level transitive. Equivalently, $f\in\Z[x]$ induces a level transitive automorphism if and only if the corresponding dynamical system $(\Z_2,f)$ is minimal (i.e., the orbit of each element of $\Z_2$ under $f$ is dense in $\Z_2$), or ergodic with respect to the Haar measure on $\Z_2$ (coinciding with the uniform Bernoulli measure on $\Z_2$ viewed as a Cantor set)~\cite[Proposition~6.5]{gns00:automata},\cite{ana:erg}.

In order to state our main result, we first review the history of the problem. The following theorem prover by Larin in~\cite{larin:tr} gives the conditions that $f$ has to  satisfy in order to be transitive mod $2^n$ for each positive integer $n$.

\begin{theorem_larin}[\cite{larin:tr}]
\label{the:larin}
A polynomial $f(x)=a_0+a_1x+ \cdots +a_tx^t \in \Z[x]$ is transitive $\mathrm{mod}\ 2^n$ for every positive integer $n$ if and only if it satisfies the following conditions:
\begin{itemize}
\item[(i)] $a_0 \equiv 1 \pmod2$
\item[(ii)] $a_1 \equiv 1 \pmod2$
\item[(iii)] $a_3+a_5+a_7+ \cdots \equiv 2a_2 \pmod4$
\item[(iv)] $a_4+a_6+a_8+ \cdots \equiv a_1+a_2-1 \pmod4$
\end{itemize}
\end{theorem_larin}


Rivest in~\cite{riv:perm} (see an alternative proof in~\cite{sun:perm}) derived the following conditions that are necessary and sufficient for a polynomial $f\in\Z[x]$ to induce a permutation of each level of $\{0,1\}^*$ and hence an automorphism of the tree.

\begin{theorem_permutational}[\cite{riv:perm}]
A polynomial $f(x)=a_0+a_1x+ \cdots +a_tx^t$ in $\Z[x]$ induces a permutation of $\ZZ$ for each positive integer $n$ if and only if it satisfies the following conditions:
\begin{itemize}
\item[(i)] $a_1 \equiv 1 \pmod2$
\item[(ii)] $a_2+a_4+a_6+ \cdots \equiv 0 \pmod2$
\item[(iii)] $a_3+a_5+a_7+ \cdots \equiv 0 \pmod2$
\end{itemize}
\end{theorem_permutational}

Using Theorem~\ref{the:sec} we study level transitivity of a permutational polynomial $f$ by counting the number of nontrivial actions of the sections of $f$ in each level. In~\cite[Proposition (4.6)]{bass_or:cyclic}, it was shown that $f$ acts level transitively on the rooted binary tree if and only if the number of nontrivial actions of the sections of $f$ in each level of $\{0,1\}^*$ is odd. Using this fact, we determine the conditions that $f$ has to meet in order to be level transitive. These conditions are summarized in the following theorem which is the main result of the paper.
\begin{theorem_main}
Let $f(x)=a_0+a_1x+ \cdots +a_tx^t$ be a permutational polynomial acting on the rooted binary tree. Then this action is level transitive if and only if the following conditions hold:
\begin{itemize}
\item[(i)] $a_0 \equiv 1 \pmod2$
\item[(ii)] $2a_2 \equiv a_3+a_5+ \cdots \pmod4$
\item[(iii)] $a_2+a_1-1 \equiv a_4+a_6+ \cdots \pmod4$
\end{itemize}
\end{theorem_main}
Combining the conditions of Theorems~\ref{the:riv} and~\ref{the:main}, we obtain the conditions of Theorem~\ref{the:larin} using a completely different approach. We hope that this new tool can be utilized to attack other problems from $p$-adic analysis. For example, the suggested approach may work to characterize ergodicity of polynomials acting on $\Z_p$ for $p>2$.

The structure of the paper is as follows. In Section~\ref{sec:pre} we set up necessary notation regarding rooted trees and their automorphisms. Section~\ref{sec:perm} describes endomorphisms and automorphisms of rooted trees arising from polynomial actions on the ring of $d$-adic integers. Finally, Section~\ref{sec:trans} contains the main result on the conditions for a permutational polynomial to act level transitively on the binary tree, which is equivalent to an ergodic action of the polynomial on $\Z_2$.

\textbf{Acknowledgement.} The authors would like to thank Zoran \v Sun\'ic and Said Sidki for fruitful and motivating discussions during the preparation of the manuscript.

\section{Preliminaries}
\label{sec:pre}
We start this section with the notation and terminology that will be used throughout the paper. A \emph{tree} is a connected graph with no cycles. A \emph{rooted tree} is a tree with one vertex selected to be the root. In any connected graph, a metric (called combinatorial metric) can be defined such that the distance between any pair of vertices is the number of edges in a shortest path (geodesic) connecting them. The $n^{th}$ level of a rooted tree is defined as the set of vertices whose distance from the root is $n$. Since the tree has no cycles, for each vertex $v$ of the $n^{th}$ level there is only one path from the root to $v$. The vertex $u$ in this path that lies in level $n-1$ is called the parent of $v$. The vertex $v$ is called a child of $u$. Hence every vertex except for the root has exactly one parent and may have some children.

A rooted tree is said to be \emph{$d$-regular} or \emph{$d$-ary} if there exists some positive integer $d$ such that each vertex of the tree has exactly $d$ children. So all these trees have infinitely many levels. In case $d=2$, such a tree is called a \emph{rooted binary tree} which represents the main interest in our paper. We will always visualize those trees such that they grow from top to bottom. So the root is the highest vertex and the children of each vertex $v$ are located right under $v$.

We will label the vertices of a rooted $d$-ary  tree by finite words over a finite alphabet $X=\{0,1, \ldots ,d-1\}$. Equivalently the set $X^*$ of all finite words over $X$ can be given the structure of a rooted $d$-ary tree by declaring that $v$ is adjacent to $vx$ for each $v \in X^*$ and $x \in X$. Thus the empty word corresponds to the root and for each positive integer $n$ the set $X^n$ corresponds to the $n^{th}$ level of the tree. The example of the rooted binary tree is shown in Figure~\ref{fig:bin_tree}.
\begin{figure} [h!]
\begin{center}
\epsfig{file=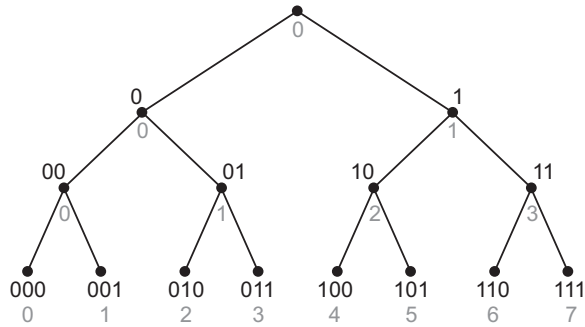,width=3in}
\caption{Standard Numbering of vertices in the binary tree\label{fig:bin_tree}}
\end{center}
\end{figure}

We will identify the $n^{th}$ level of $X^*$ with the ring $\Zd$ by identifying a vertex $x_0x_1 \ldots x_{n-1} \in X^n$ with $x_0+dx_1+ \cdots +d^{n-1}x_{n-1} \in \Zd$. For example, the vertices $00$, $01$, $10$ and $11$ of the second level of the rooted binary tree are identified with $0$, $2$, $1$ and $3$, respectively as shown in Figure~\ref{fig:bin_tree_mix}. Moreover, the boundary of the tree can be naturally identified with the ring $\Z_d$ of $d$-adic integers.

\begin{figure} [h!]
\begin{center}
\epsfig{file=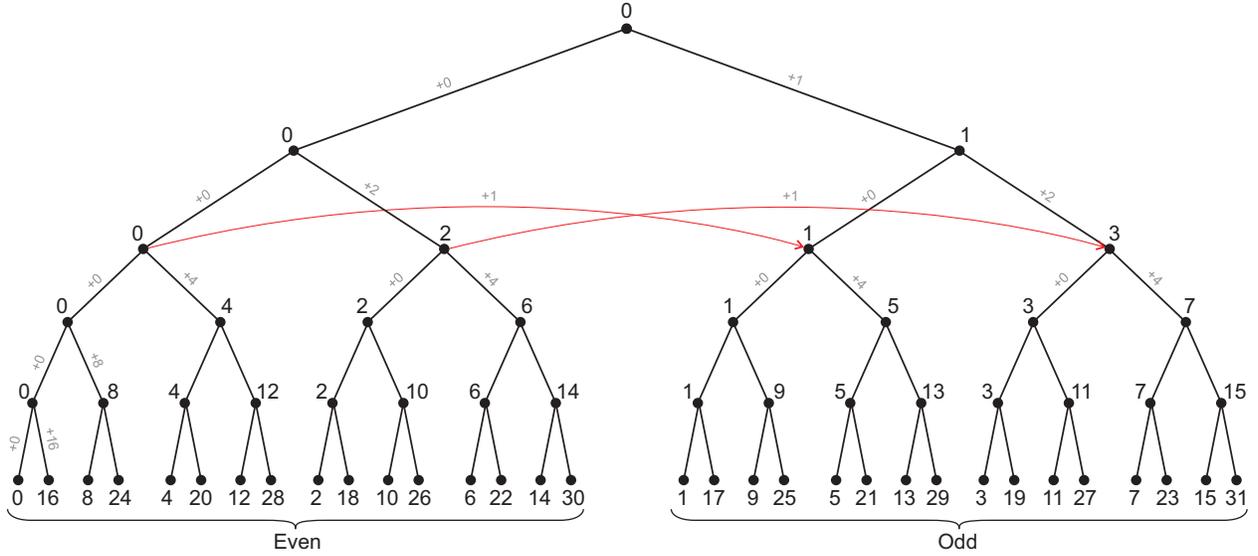,width=\textwidth}
\caption{Dyadic Numbering of vertices in the binary tree\label{fig:bin_tree_mix}}
\end{center}
\end{figure}
The way we identified the $n^{th}$ level of $X^*$ with $\Zd$ may not be the most natural way. A more natural way to do that is to identify a vertex $x_0x_1 \ldots x_{n-1}$ with its $d$-ary expansion $x_{n-1}+dx_{n-2}+ \cdots +d^{n-1}x_0$. So the vertices $00$, $01$, $10$ and $11$ of the second level of the binary tree are identified with $0$, $1$, $2$ and $3$, respectively, as shown in Figure~\ref{fig:bin_tree}. However, we adopt the first identification so that mappings induced by polynomials on $X^*$ preserve the adjacency relation as we will see later.

\begin{definition}
An \emph{endomorphism} of $X^*$ is a map from the set of vertices of $X^*$ to itself which preserves the adjacency relation. If such a map is bijective, it is called an automorphism of $X^*$.
\end{definition}
An automorphism of $X^*$ preserves the degree of each vertex as well as the distance from each vertex to the root. Since the root is the only vertex with degree $d$, it is invariant under all the tree automorphisms. Also the levels of the tree are invariant under automorphisms since the distance is preserved. The group of all automorphisms of $X^*$ is denoted by $\Aut X^*$.

Another important concept that we want to introduce here is the definition of the section of an endomorphism of $X^*$ at a vertex.
\begin{definition}
Let $h$ be an endomorphism of $X^*$ and $x \in X$. For any word $v \in X^*$, we have $h(xv)=h(x)v'$ for some $v' \in X^*$. The map $h|_x \colon X^* \rightarrow X^*$ given by
\begin{equation}
\label{equ:sec}
h|_x(v)=v'
\end{equation}
clearly defines an endomorphism of $X^*$ and is called the \emph{section} of $h$ at vertex $x$. Inductively we can define the section of $h$ at a vertex $v=x_1x_2 \ldots x_n$ as $h|_v=h|_{x_1}|_{x_2} \ldots |_{x_n}$.
\end{definition}
In order to fully define the action of an endomorphism $h$ on $X^*$, we need to specify its action on the first level as well as its sections at the vertices of the first level. To see the action of $h$ on the second level, we do the same with the sections of $h$ at the vertices of the first level and so on. In case of an endomorphism $h$ of a rooted binary tree, the action of $h$ on the first level is either trivial or is a switch of $0$ and $1$. Another language to define an endomorphism of $X^*$ is to give its wreath recursion which also specifies its action on the first level and its sections at the vertices of the first level. This language makes computations easier. Since we will not do any computations with endomorphisms, this approach will not be used here.

We can visualize the action of $h$ on $X^*$ using what is called the \emph{full portrait} of $h$. The full portrait of $h$ is a labeled infinite rooted $d$-ary tree with the root labeled by the name of the endomorphism $h$ and each vertex $v$ labeled by $h|_v$. Under each vertex $v$, we usually write the name of the mapping that $h|_v$ defines on the first level of the subtree hanging from $v$. In the case of a rooted binary tree, we draw a little arc (called switch) connecting the two edges hanging down from  $v$ if $h|_v$ acts nontrivially on the first level of the subtree hanging from $v$. If there is no switch, it means the action is trivial.

One of the basic examples of automorphisms of the rooted binary tree is the adding machine which we will denote throughout the paper by $\tau$. It gets its name from the fact that its action on the boundary of the tree, which is identified by the ring $\Z_2$, is equivalent to adding one to the input. The sections of $\tau$ at the vertices $0$ and $1$ are respectively $e$ and $\tau$, where $e$ is the identity automorphism. And $\tau$ acts nontrivially on the first level. The full portrait of $\tau$ is shown in Figure~\ref{fig:adding_machine}.

\begin{figure}
\begin{center}
\epsfig{file=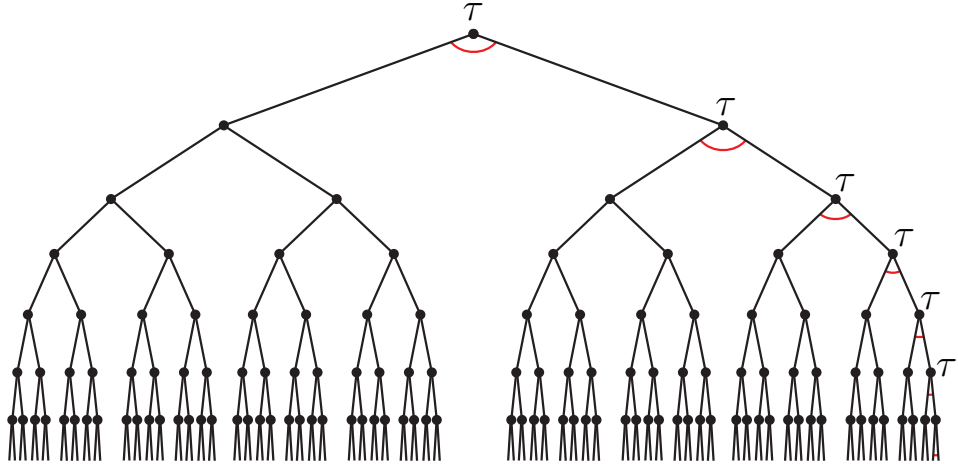,width=5in}
\caption{The full portrait of the adding machine\label{fig:adding_machine}}
\end{center}
\end{figure}

The next definition introduces the notion of level transitivity which is the core concept of our paper. Also a necessary and sufficient condition for an automorphism to act level transitively on a rooted binary tree is provided in the next theorem, that is a partial case of a more general result~\cite[Proposition (4.6)]{bass_or:cyclic}.
\begin{definition}
An automorphism  $h$ of $X^*$ is said to act \emph{level transitively} on $X^*$ if it acts transitively on each level of $X^*$.
\end{definition}
\begin{proposition}
\label{pro:tra}
Let $h$ be an automorphism of the rooted binary tree $\{0,1\}^*$. Then $h$ acts level transitively if and only if the full portrait of $h$ has an odd number of switches (nontrivial actions) in each level including the zeroth level.
\end{proposition}

\begin{remark}
The proof of this theorem involves induction on the level of the tree and is quite short. We leave it as an exercise and only point out that even though this is (the only) external result that we use in this paper, its proof can be included here essentially at no extra cost, so it does not affect our claim that our proof is self contained (modulo the result of Rivest~\cite{riv:perm} on the characterization of permutational polynomials).
\end{remark}

\begin{example}
As we can see in Figure~\ref{fig:adding_machine}, the adding machine $\tau$ has exactly one switch in each level. So the last theorem asserts that $\tau$ acts level transitively on the rooted binary tree.
\end{example}

\section{Endomorphisms of Rooted Trees Arising from Polynomials over $\Z$}
\label{sec:perm}
For a fixed integer $d\geq 2$, each polynomial $f(x) \in \Z[x]$ induces mappings $f_n \colon \Zd\rightarrow\Zd$ for all positive integers $n$ by taking the evaluation map modulo $d^n$. Identifying the $n^{th}$ level of the rooted $d$-ary tree $X^*$ with the ring $\Zd$, the polynomial $f$  gives rise to a mapping on the whole tree. In the next proposition, we show that this mapping is always an endomorphism of the rooted tree as it preserves the root and the adjacency relation. If for some $d$ the mapping $f_n$ is a bijection for each $n$, then $f$ induces a permutation of each level of the rooted $d$-ary tree $X^*$ and hence induces an automorphism of $X^*$. From now on, we will use the term \emph{permutational polynomial} to denote a polynomial in $\Z[x]$ that induces an automorphism of $X^*$ and to denote the automorphism it induces on $X^*$ as well. The same letter will be used to refer to both functions and no confusion will arise.

\begin{proposition}
\label{pro:distinct}
Let $f(x) \in \Z[x]$. Then $f$ induces an endomorphism on the rooted $d$-ary tree $X^*$. Moreover, different polynomials over $\Z$ induce different endomorphisms on $X^*$.
\begin{proof}
We pick two adjacent vertices $u$ and $v$ of the tree such that $u$ is the parent of $v$. So $$u=x_0+dx_1+ \cdots +d^{n-1}x_{n-1} \in X^n$$ and $$v=x_0+dx_1+ \cdots +d^{n-1}x_{n-1}+d^nx_n \in X^{n+1}$$ for some $x_0, \ldots ,x_n \in \{0,1, \ldots ,d-1\}$ and $n \geq 0$.
Let
\[f(u)=y_0+dy_1+ \cdots +d^{n-1}y_{n-1}\] and
\[f(v)=y'_0+dy'_1+ \cdots +d^{n-1}y'_{n-1}+d^ny'_n,\]
where $y_0, \ldots ,y_{n-1},y'_0, \ldots ,y'_n \in \{0,1, \ldots ,d-1\}$.
Since $u \equiv v \pmod {d^n}$, it follows that $f(u) \equiv f(v) \pmod {d^n}$. Thus $$y_0+dy_1+ \cdots +d^{n-1}y_{n-1}=y'_0+dy'_1+ \cdots +d^{n-1}y'_{n-1}$$ By uniqueness of the $d$-ary expansion, we obtain $y_i=y'_i$ for $i=0,\ldots,n-1$. Hence $f(u)$ is the parent of $f(v)$. This means that $f$ preserves the adjacency relation so it is an endomorphism of $X^*$.

Now we consider two different polynomials $g,h \in \Z[x]$. We can find an integer $k$ such that $g(k)\neq h(k)$. Let $N$ be the smallest positive integer such that $g(k) \not \equiv h(k) \pmod{d^N}$. Then the actions of $g$ and $h$ on the $N^{th}$ level of $X^*$ are different.
\end{proof}
\end{proposition}

Our next goal is to completely describe endomorphisms induced by polynomials by explicitly describing their sections at all the vertices of $X^*$. Before we proceed to the next theorem, we need to introduce some basic notation to make our expressions less cumbersome.
\vspace{2mm}

\noindent\textbf{Notation.} Given two integers $d>0$ and $a$, we can use the division algorithm to find two unique integers $q$ and $r$ such that $a=dq+r$ where $r \in \{0,1, \ldots ,d-1\}$. We will adopt the notation $Q_d(a)=q$ and $R_d(a)=r$ i.e., $Q_d(a)=a \mathop{\mathrm{div}} d$ and $R_d(a)$ is the remainder of $a$ when divided by $d$. So for each integer $a$ we always have
\begin{equation}
\label{equ:Q_and_r}
a=R_d(a)+dQ_d(a).
\end{equation}

\begin{theorem}
\label{the:sec}
Given a polynomial $f(x)=a_0+a_1x+ \cdots +a_tx^t \in \Z[x]$ inducing an endomorphism of $X^*$, the image of a vertex $x_0 \in X$ under the induced endomorphism is $R_d(f(x_0))$ and the section of $f$ at $x_0$ is again a polynomial given by the equation:
\begin{equation}
\label{equ:secf}
f|_{x_0}(x)=Q_d(f(x_0))+\sum_{i=1}^{t}\frac{f^{(i)}(x_0)}{i!}d^{i-1}x^i
\end{equation}
\end{theorem}

\begin{proof}
We pick a vertex $u=x_0+dx_1+d^2x_2+ \cdots +d^nx_n$ which has $x_0$ as a prefix. We can write $u$ as $u=x_0+d\cdot v$, where $v=x_1+dx_2+ \cdots +d^{n-1}x_n$ corresponds to the suffix of the word $u$ as in equation~\eqref{equ:sec}.  Using the Taylor expansion of $f$ about $x_0$ and the fact that according to equality~\eqref{equ:Q_and_r}, $f(x_0)=R_d(f(x_0))+dQ_d(f(x_0))$ we obtain $$f(u)=f(x_0+d\cdot v)=\sum_{i=0}^{t}\frac{f^{(i)}(x_0)}{i!}d^iv^i=R_d(f(x_0))+d\cdot\left(Q_d(f(x_0))+\sum_{i=1}^{t}\frac{f^{(i)}(x_0)}{i!}d^{i-1}v^i\right)$$ Therefore, $$f(x_0)=R_d(f(x_0+d\cdot v))=R_d(f(x_0)),$$
where with a slight abuse of notation in the left hand side $x_0$ and $f(x_0)$ denote vertices of the first level of $X^*$, and in the right hand side $x_0$ is an element of $\Z/(d\Z)$. And finally,
$$f|_{x_0}(v)=Q_d(f(x_0+d\cdot v))=Q_d(f(x_0))+\sum_{i=1}^{t}\frac{f^{(i)}(x_0)}{i!}d^{i-1}v^i$$
\end{proof}

\begin{remark}
\label{rem:deg}
Equation~\eqref{equ:secf} immediately implies that all the sections of $f$ are polynomials of the same degree as the degree of $f$. Also since $f^{(t)}(x)=t!a_t$, all the sections of $f$ at the same level $n$ have the same leading coefficient $(d^{t-1})^na_t$.
\end{remark}

\begin{example}
The sections of the polynomial $f(x)=1+3x+2x^2$ at the vertices of the first three levels of the tree are depicted in Figure~\ref{fig:sect_example}.
\end{example}

\begin{figure}
\begin{center}
\epsfig{file=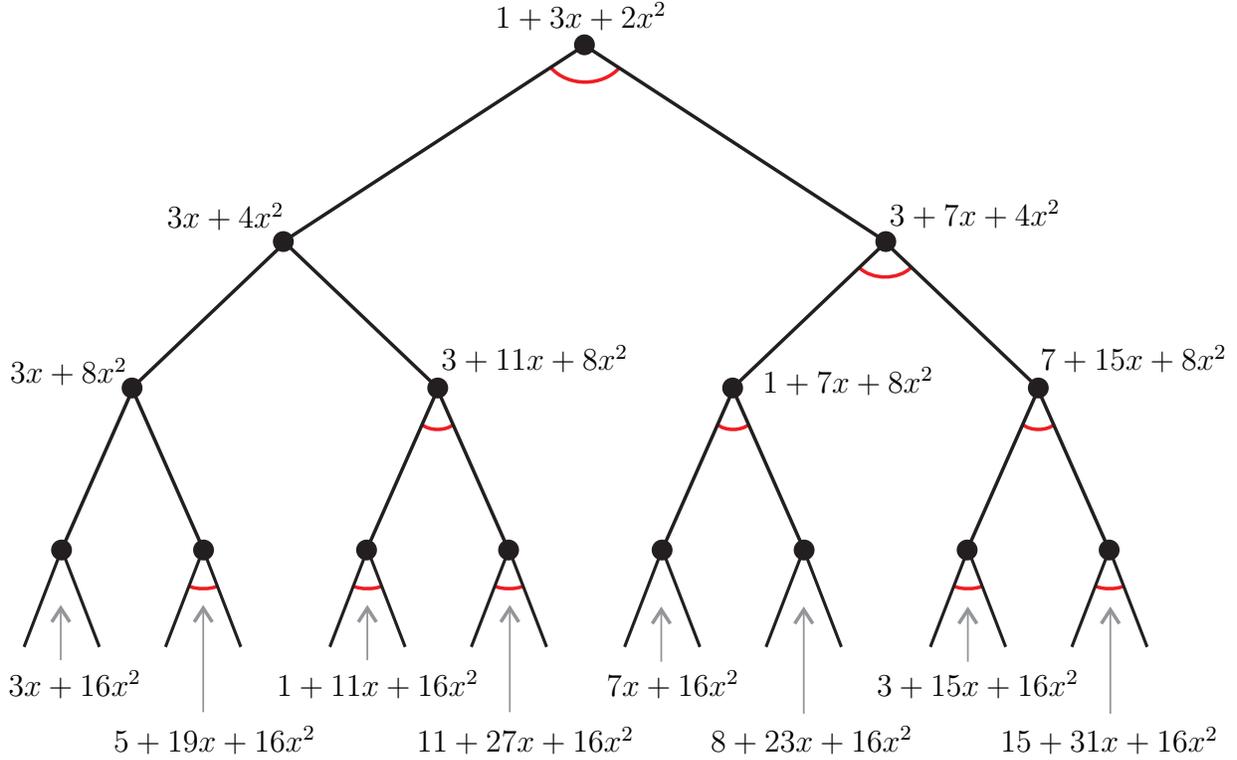}
\caption{Sections of an automorphism of $\{0,1\}^*$ induced by a polynomial $f(x)=1+3x+2x^2$\label{fig:sect_example}}
\end{center}
\end{figure}

\begin{proposition}
A polynomial $f$ acting on $X^*$ has finitely many distinct sections if and only if it is linear.
\end{proposition}

\begin{proof}
First we show that a linear polynomial $f(x)=ax+b$ acting on $X^*$ has finitely many sections. From equation~\eqref{equ:secf}, we see that the section of $f$ at a vertex $x_0 \in X$ is given by $f|_{x_0}(x)=ax+Q_d(b+x_0a)$. Hence all the sections of $f$ are linear polynomials with the same leading coefficient $a$. So the number of sections of $f$ is bounded up by the number of distinct constant terms of these sections. Since the constant term can be written as $Q_d(b+a+ \cdots +a)$ where we have exactly $x_0 \leq d-1$ summands equal to $a$, it is enough to notice that for any collection of integers $n_1,n_2, \ldots ,n_d$ we have $\displaystyle{|Q_d(n_1+n_2+ \cdots +n_d)|} \leq \max\{|n_1|,|n_2|, \ldots ,|n_d|\}$. Therefore, any section of $f$ has the form $ax+c$ where $|c|\leq \max\{|a|,|b|\}$.

The fact that any nonlinear polynomial $f(x)=a_0+a_1x+ \cdots +a_tx^t$ acting on $X^*$ has infinitely many sections follows immediately from Remark~\ref{rem:deg} and Proposition~\ref{pro:distinct}.
\end{proof}

Now we turn our attention to the permutational polynomials inducing automorphisms of $X^*$. First of all, we recall the definition.

\begin{definition}
A polynomial $f(x)\in\Z[x]$ is said to be \emph{$d$-permutational} (or, simply, \emph{permutational} if $d$ is clear from the context) if for each $n\geq1$ the mapping $f_n\colon\Z/(d^n\Z)\to \Z/(d^n\Z)$ induced by the evaluation homomorphism is a permutation.
\end{definition}

According to Proposition~\ref{pro:distinct} each permutational polynomial induces an automorphism of the $d$-ary tree $X^*$. The following simple remark follows immediately from the definition.

\begin{remark}
The sections of a permutational polynomial acting on $X^*$ are again permutational.
\end{remark}

For each $d \geq 2$, the set of all linear $d$-permutational polynomials obviously forms a group under the operation of composition. However, the set of all $d$-permutational polynomials does not form a group as it is not closed under taking inverses. In fact, it has the structure of a cancellative monoid as shown in the next proposition.

\begin{proposition}
\label{pro:monoid}
For each $d\geq 2$, the set of all $d$-permutational polynomials forms a cancellative monoid under the operation of composition.
\end{proposition}

\begin{proof}
It is clear that the composition of two $d$-permutational polynomials is again a polynomial and that it induces an automorphism of the rooted $d$-ary tree $X^*$. The polynomial $\mathrm{id}(x)=x$ plays the role of the identity as it clearly induces the identity automorphisms of the tree. so we have the structure of a monoid.

The inverse of an automorphism of $X^*$ induced by a permutational polynomial always exists but it cannot always be induced by a polynomial. Although this is the case for linear permutational polynomials, the inverse of a nonlinear permutational polynomial cannot be induced by a polynomial. If this were true, we would have two polynomials not both linear acting on the ring $\Z_d$ (the boundary of $X^*$) with a trivial composition which is impossible. Still the existence of the inverse makes cancellation legitimate.
\end{proof}

\begin{remark}
If we consider the action of a $d$-permutational polynomial $f$ on some specific level $n$, identified with $\Zd$, the inverse of $f$ on $\Zd$ can always be induced by a permutational polynomial as shown in~\cite{sun:perm}.
\end{remark}

For the rest of the paper, we will consider only permutational polynomials acting on the rooted binary tree $\{0,1\}^*$. The next theorem, introduced by Rivest in \cite{riv:perm}, determines the conditions under which a polynomial $f(x) \in \Z[x]$ induces a permutation of $\ZZ$ for each $n$ and hence an automorphism of $\{0,1\}^*$ (i.e., is a $2$-permutational polynomial).
\begin{theorem}[\cite{riv:perm,sun:perm}]
\label{the:riv}
A polynomial $f(x)=a_0+a_1x+ \cdots +a_tx^t \in \Z[x]$ induces a permutation of $\ZZ$ for each $n$ if and only if it satisfies the following conditions:
\begin{itemize}
\item[(i)] $a_1 \equiv 1 \pmod2$
\item[(ii)] $a_2+a_4+a_6+ \cdots \equiv 0 \pmod2$
\item[(iii)] $a_3+a_5+a_7+ \cdots \equiv 0 \pmod2$
\end{itemize}
\end{theorem}
The theorem does not put any restriction on the constant term $a_0$ of the polynomial $f$. Assuming that $f$ satisfies the conditions of the theorem i.e., it induces a permutation of $\ZZ$ and that $a_0$ is even (odd), all the evens of $\ZZ$ are mapped to evens (odds) and  all the odds of $\ZZ$ are mapped to odds (evens). A special case of the theorem is when $f(x)=ax+b$ i.e., when $f$ is linear. It defines a permutation of $\ZZ$ if and only if $a$ is odd.

From now on, a permutational polynomial will always mean a $2$-permutational polynomial. Also we will drop the $2$ subscript in the notation of the function $Q_2$ and write $Q(a)=a \mathop{\mathrm{div}} 2$. According to equation~\eqref{equ:secf}, the sections of a permutational polynomial $f(x)=a_0+a_1x+ \cdots +a_tx^t$ acting on $\{0,1\}^*$ at the vertices $0$ and $1$ are given by the two equations:
\begin{equation}
\label{equ:sec0}
f|_0(x)=Q(a_0)+a_1x+2a_2x^2+ 2^2a_3x^3\cdots +2^{t-1}a_tx^t
\end{equation}
and
\begin{multline}
\label{equ:sec1}
f|_1(x)=Q(a_0+a_1+ \cdots +a_t)+(a_1+2a_2+ \cdots +ta_t)x\\
+\bigl(2 \cdot 1a_2+3 \cdot 2a_3+ \cdots +t(t-1)a_t\bigr)x^2+4\sum_{i=3}^{t}\frac{f^{(i)}(1)}{i!}2^{i-3}x^i
\end{multline}
In the case of a linear polynomial $f(x)=ax+b$, we have $f|_0(x)=ax+Q(b)$ and $f|_1(x)=ax+Q(a+b)$. For example, the adding machine $\tau$, which was introduced in the preliminary section, can be represented by a linear permutational polynomial $\tau(x)=x+1$. The sections of $\tau$ at $0$ and $1$ are respectively $x$ and $x+1$.

According to Theorem~\ref{the:riv}, the group of all linear $2$-permutational polynomials is isomorphic to the group
$$\left\{\left( {\begin{array}{ccccc}
a & b \\
0 &1
\end{array} } \right) \mid a,b \in \Z,\,\, a \,\, \text{odd}\, \right\}.$$
A generating set of this group is $\{p,q_{-1},q_3,q_5,q_7, \ldots \}$, where $p(x)=x+1$ and $q_m(x)=mx$ for $m=-1,3,5,7, \ldots $. All the generators have infinite orders except for $q_{-1}$ which is an involution. It was shown in~\cite{bartholdi_s:bsolitar} that for each $m=-1,3,5,7, \ldots $, the subgroup generated by $p$ and $q_m$ is the Baumslag-Solitar group $BS(1,m)=\langle p,q_m \mid q_mp^m=pq_m \rangle$, where we adopt the convention that when the expression $fg$ is used to denote the composition of two functions it means the function $f$ acts first. In the same paper, Bartholdi and \v Sun\'ic considered sections of some linear polynomials acting on $\{0,1\}^*$.

We end this section by introducing one more notation and stating a simple lemma that will be used in Section~\ref{sec:trans} to simplify the proof of the main result.
\vspace{2mm}

\noindent\textbf{Notation.} Suppose $f(x)=a_0+a_1x+ \cdots +a_tx^t \in \Z[x]$ is a permutational polynomial. Then it uniquely defines the following integers $k_f$, $m_f$ and $n_f$ such that
\begin{equation}
\begin{array}{l}
a_1=2k_f+1,\\
a_2+a_4+\cdots=2m_f,\\
a_3+a_5+\cdots=2n_f.
\end{array}
\end{equation}
\vspace{2mm}

The following lemma, in particular, shows that not every permutational polynomial can be a section of another permutational polynomial.

\begin{lemma}
\label{lem:n_f}
Let $f(x)=a_0+a_1x+ \cdots +a_tx^t$ be a permutational polynomial acting on the rooted binary tree $\{0,1\}^*$. Then for each non-root vertex $v\in \{0,1\}^*$, the permutational polynomial $f|_v(x)=b_0+b_1x+ \cdots +b_tx^t$ corresponding to the section of $f$ at $v$, satisfies $n_{f|_v} \equiv 0\pmod 2$ and $b_2 \equiv 0 \pmod2$.
\end{lemma}

\begin{proof}
Clearly, it is enough to check the conditions only for vertices of the first level. But this trivially follows from equations~\eqref{equ:sec0} and~\eqref{equ:sec1}. We only have to notice that all the coefficients at $x^i$ for $i\geq 3$ in both $f|_0$ and $f|_1$ are divisible by 4 and that the coefficients at $x^2$ in both $f|_0$ and $f|_1$ are divisible by $2$.
\end{proof}

\section{Level Transitivity of Permutational Polynomials Acting on a Rooted Binary Tree}
\label{sec:trans}
We start this section by presenting a couple of basic number theoretic facts that will be used many times in the proof of our main theorem. We introduce a bunch of lemmas without proofs. The proofs are straight forward so we leave them as simple exercises. The first lemma gives some properties of the function $Q(x)$ which was defined in the last section. All the other lemmas follow from the fact that for any positive integer $n$ we have $ax \equiv x \pmod{n}$ if $a$ and $x$ are integers and $a \equiv 1 \pmod{n}$.
\begin{lemma}
\label{lem:Q}
For any three integers $x$, $m$ and $n$, we have:
\begin{itemize}
\item[(i)] $Q(x+2n)=Q(x)+n$
\item[(ii)] $Q(x)+Q(x+1)=x$
\item[(iii)] $Q(x)+Q(x+1+2n)=x+n$
\item[(iv)] $Q(x+4n) \equiv Q(x) \pmod2$
\item[(v)] $Q(x)+Q(x+1+4n) \equiv x \pmod2$
\item[(vi)] If $n \equiv m \pmod4$, then $Q(x+n) \equiv Q(x+m) \pmod2$
\end{itemize}
\end{lemma}

\begin{lemma}
\label{lem:odd}
Let $a_1,a_2, \ldots ,a_n$ be a collection of odd numbers. Then for any sequence of integers $x_1,x_2, \ldots ,x_n$, we have $a_1x_1+a_2x_2+ \cdots +a_nx_n \equiv x_1+x_2+ \cdots +x_n \pmod2$ i.e., $a_1x_1+a_2x_2+ \cdots +a_nx_n$ has the same parity as $x_1+x_2+ \cdots +x_n$.
\end{lemma}

\begin{lemma}
\label{lem:mod4}
Let $a_1,a_2, \ldots ,a_n$ and $x_1,x_2, \ldots ,x_n$ be integers such that $a_i \equiv 1 \pmod4$ for $i=1,2, \ldots ,n$. Then $a_1x_1+a_2x_2+ \cdots +a_nx_n \equiv x_1+x_2+ \cdots +x_n \pmod4$.
\end{lemma}

\begin{lemma}
\label{lem:squares}
Let $a_1,a_2, \ldots ,a_n$ be a collection of odd numbers and $x_1,x_2, \ldots ,x_n$ be any sequence of integers. Then $a_1^2x_1+a_2^2x_2+ \cdots +a_n^2x_n \equiv x_1+x_2+ \cdots +x_n \pmod4$.
\end{lemma}
Now it is time to introduce our main theorem that determines the conditions under which a permutational polynomial acts level transitively on the rooted binary tree $\{0,1\}^*$. Obviously a permutational polynomial acts nontrivially on the first level of the tree if and only if its constant term is odd. According to Proposition~\ref{pro:tra}, a permutational polynomial $f$ acts level transitively on $\{0,1\}^*$ if and only if in each level of the tree, the number of sections of $f$ with odd constant terms is odd, or equivalently, the sum of these constant terms is odd. The next proposition determines the conditions that a linear permutational polynomial has to meet in order to act level transitively on $\{0,1\}^*$. The general result is given by Theorem~\ref{the:main} which we prove by induction on level. The proof of the linear case is provided first so that the idea of the induction becomes clear and then the general case is considered. Actually the proof given here is essentially different from the proofs of similar results introduced in~\cite{ana:erg} and in~\cite{larin:tr}.

\begin{proposition}
\label{pro:lin}
Let $f(x)=ax+b$ be a permutational polynomial acting on the rooted binary tree $\{0,1\}^*$. Then this action is level transitive if and only if the following conditions hold:
\begin{itemize}
\item[(i)] $b \equiv 1 \pmod2$
\item[(ii)] $a \equiv 1 \pmod4$
\end{itemize}
\end{proposition}

\begin{proof}
We first show that the conditions are necessary. If condition $(i)$ is not satisfied, then $f$ does not act transitively on the first level. So we assume that condition $(i)$ is satisfied but condition $(ii)$ is not. We can also use the fact that $a$ is odd since $f$ is permutational. Thus we can write $b=2s+1$ and $a=2r+1$, where $s$ and $r$ are integers and $r$ is odd. The sections of $f$ at $0$ and $1$ are respectively $ax+Q(b)=ax+s$ and $ax+Q(a+b)=ax+r+s+1$. The sum of the two constant terms is $2s+r+1$ which is even. Therefore, $f$ does not act transitively on the second level.

To prove sufficiency of the conditions, we use induction on level. First we can write $a=4k+1$ for some integer $k$ which follows from condition $(ii)$. Since $b$ is odd, $f$ acts transitively on the first level. Now assume that the sections of $f$ at some level $n \geq 0$ are respectively $ax+\alpha_1, \ldots ,ax+\alpha_m$, where $m=2^n$ and suppose that $\alpha_1+ \cdots +\alpha_m \equiv 1 \pmod2$ which serves as the induction hypothesis. The sections of $f$ at level $n+1$ are respectively $ax+\sigma_1,ax+\sigma_2, \ldots ,ax+\sigma_{2m-1},ax+\sigma_{2m}$, where $\sigma_1=Q(\alpha_1),$ $\sigma_2=Q(\alpha_1+4k+1),$ $ \ldots ,$ $\sigma_{2m-1}=Q(\alpha_m),$ $\sigma_{2m}=Q(\alpha_m+4k+1)$. Hence the sum of constant terms of the sections of level $n+1$ is $(\sigma_1+\sigma_2)+ \cdots +(\sigma_{2m-1}+\sigma_{2m})=\left(Q(\alpha_1)+Q(\alpha_1+4k+1)\right)+ \cdots +\left(Q(\alpha_m)+Q(\alpha_m+4k+1)\right) \equiv \alpha_1+ \cdots +\alpha_m \equiv 1 \pmod2$, where we have used Lemma~\ref{lem:Q} part $(v)$. The proof is now complete.
\end{proof}

The adding machine $\tau(x)=x+1$ as well as all its odd powers $\tau^r(x)=x+r$, where $r$ is odd, satisfy the conditions of the last proposition. So they act level transitively on the rooted binary tree $\{0,1\}^*$.
\begin{theorem}
\label{the:main}
Let $f(x)=a_0+a_1x+ \cdots +a_tx^t$ be a permutational polynomial acting on the rooted binary tree $\{0,1\}^*$. Then this action is level transitive if and only if the following conditions hold:
\begin{itemize}
\item[(i)] $a_0 \equiv 1 \pmod2$
\item[(ii)] $2a_2 \equiv a_3+a_5+ \cdots \pmod4$
\item[(iii)] $a_2+a_1-1 \equiv a_4+a_6+ \cdots \pmod4$
\end{itemize}
\begin{proof}
To show that the given conditions are necessary, we first notice that if condition $(i)$ is not satisfied then $f$ does not act transitively on the first level. So we assume that condition $(i)$ is satisfied and hence we can write $a_0=2s+1$ for some integer $s$. Since $f$ is permutational, it satisfies the three conditions of Theorem~\ref{the:riv}. These three conditions could be written respectively as $a_1=2k+1,$ $a_2+a_4+ \cdots =2m$ and $a_3+a_5+ \cdots =2n$ for some integers $k,m$ and $n$. By adding $a_2$ to both sides in condition $(iii)$ and doing some simple algebraic manipulation, conditions $(ii)$ and $(iii)$ can be rewritten as:
\begin{itemize}
\item[($ii'$)] $a_2 \equiv n \pmod2$
\item[($iii'$)] $k+a_2 \equiv m \pmod2$
\end{itemize}
Equations~\eqref{equ:sec0} and~\eqref{equ:sec1} tell us that the constant terms of the sections of $f$ at $0$ and $1$ are respectively $Q(a_0)$ and $Q(a_0+a_1+ \cdots +a_t)$. Their sum is $Q(a_0)+Q(a_0+a_1+ \cdots +a_t)=Q(a_0)+Q(a_0+2k+1+2m+2n)=a_0+k+m+n=2s+(k+n)+m+1$, where we have used Lemma~\ref{lem:Q} part $(iii)$. For $f$ to act transitively on the second level, we must have
\begin{equation}
\label{equ:trans2}
(k+n) \equiv m \pmod2.
\end{equation}
If exactly one of the two conditions $(ii')$ and $(iii')$ is satisfied, then $(k+n) \not\equiv m \pmod2$. However, if the two conditions are both not satisfied, we have $(k+n) \equiv m \pmod2$ so $f$ does act transitively on the second level. We will show that in the latter case, $f$ does not act transitively on the third level by proving that the sum of constant terms of the sections of $f$ at the second level is even.

From equations~\eqref{equ:sec0} and \eqref{equ:sec1}, we can infer that the constant terms of the sections of $f$ at the vertices $00,01,10$ and $11$ are respectively:
\begin{itemize}
\item $Q(Q(a_0))=Q(s)$
\item $Q(s+a_1+2a_2+ \cdots +2^{t-1}a_t)=Q(s+2k+1+2a_2+ \cdots +2^{t-1}a_t)$
\item $Q(Q(a_0+a_1+ \cdots +a_t))=Q(Q(2s+1+2k+1+2m+2n))=Q(1+s+k+m+n)$
\item $Q((1+s+k+m+n)+(a_1+2a_2+ \cdots +ta_t)+(2 \cdot 1a_2+3 \cdot 2a_3+ \cdots +t(t-1)a_t)+4\sum_{i=3}^{t}\frac{f^{(i)}(1)}{i!}2^{i-3})$
\end{itemize}
The sum of these constant terms modulo $2$ is
$$Q(s)+Q(s+2k+1+2a_2)+Q(1+s+k+m+n)+Q((1+s+k+m+n)+a_1+2^2a_2+3^2a_3+ \cdots +t^2a_t),$$ where we have used Lemma~\ref{lem:Q} part $(iv)$. Using part $(iii)$ of the same lemma and reusing part $(iv)$ as well, the sum modulo $2$ simplifies to
$$s+k+a_2+Q(1+s+k+m+n)+Q((1+s+k+m+n)+1+2k+(2^2a_2+4^2a_4+ \cdots)+(3^2a_3+5^2a_5+ \cdots))$$$$ \equiv s+k+a_2+Q(1+s+k+m+n)+Q((1+s+k+m+n)+1+2k+(3^2a_3+5^2a_5+ \cdots))$$$$\equiv s+k+a_2+Q(1+s+k+m+n)+Q((1+s+k+m+n)+1+2k+(a_3+a_5+ \cdots))$$$$ \equiv s+k+a_2+Q(1+s+k+m+n)+Q((1+s+k+m+n)+1+2k+2n) \pmod2$$ Where we have used Lemma~\ref{lem:squares} to obtain $3^2a_3+5^2a_5+ \cdots \equiv a_3+a_5+ \cdots $.
Again we apply part $(iii)$ of Lemma~\ref{lem:Q} to have the sum modulo $2$ as
$$s+k+a_2+(1+s+k+m+n)+k+n \equiv 2s+2k+(k+n)+m+(n+a_2+1)$$$$ \equiv ((k+n)+m)+(n+a_2+1) \equiv (k+n)+m \equiv 0 \pmod2$$ where we have used the fact that the two conditions $(ii')$ and $(iii')$ are both not satisfied.

Now we prove sufficiency of the conditions. Again we use the fact that $f$ is permutational to write $a_1=2k+1,$ $a_2+a_4+ \cdots =2m$ and $a_3+a_5+ \cdots =2n$ for some integers $k,m$ and $n$. We can easily infer that $2a_2 \equiv 2n \pmod4$ and $2n+2k \equiv 2m \pmod4$. The last equality is equivalent to equality~\ref{equ:trans2}. We will show that $f$ acts transitively on the first and second levels and then use induction for lower levels. Condition $(i)$ guarantees transitivity on the first level.

We deduced in the proof of necessity that the transitivity on the second level is equivalent to equality~\eqref{equ:trans2}, which is automatically satisfied in our case as shown above.

Before we formulate the induction hypothesis, we recall that at each vertex $v$ of $\{0,1\}^*$ the section $f|_v$ is induced by a permutational polynomial $f|_v(x)=a_0^{\f}+a_1^{\f}x+ \cdots +a_t^{\f}x^t$, so it defines the integers $k_{f|_v}$, $m_{f|_v}$, and $n_{f|_v}$ such that $a_1^{\f}=2k_{f|_v}+1$, $a_2^{\f}+a_4^{\f}+ \cdots =2m_{f|_v}$ and $a_3^{\f}+a_5^{\f}+ \cdots =2n_{f|_v}$. Also if $v$ is a non-root vertex, we have $n_{f|_v} \equiv 0 \pmod2$ and $a_2^{\f} \equiv 0 \pmod2$ by Lemma~\ref{lem:n_f}.

Let us assume that the sections of $f$ at some level $l \geq 1$ of the tree are $\varphi_1, \ldots ,\varphi_r$, where $r=2^l$. Suppose that $a_0^{(\varphi_1)}+ \cdots +a_0^{(\varphi_r)} \equiv 1 \pmod2$ which serves as our induction hypothesis. Let $\psi_1,\psi_2, \ldots ,\psi_{2r-1},\psi_{2r}$ be the sections of $f$ at level $l+1$. Then

$$\sum_{i=1}^{2r}a_0^{(\psi_i)}=\sum_{j=1}^{r}\left(a_0^{(\psi_{2j-1})}+a_0^{(\psi_{2j})}\right)$$
$$=\sum_{j=1}^{r}\left(Q\left(a_0^{(\varphi_j)}\right)+Q\left(a_0^{(\varphi_j)}+a_1^{(\varphi_j)}+ \cdots +a_t^{(\varphi_j)}\right)\right)$$
$$=\sum_{j=1}^{r}\left(Q\left(a_0^{(\varphi_j)}\right)+Q\left(a_0^{(\varphi_j)}+2k_{\varphi_j}+1+2m_{\varphi_j}+2n_{\varphi_j}\right)\right)$$
$$=\sum_{j=1}^{r}\left(a_0^{(\varphi_j)}+k_{\varphi_j}+m_{\varphi_j}+n_{\varphi_j}\right)$$
$$=\sum_{j=1}^{r}a_0^{(\varphi_j)}+\sum_{j=1}^{r}k_{\varphi_j}+\sum_{j=1}^{r}m_{\varphi_j}+\sum_{j=1}^{r}n_{\varphi_j}$$
$$\equiv 1+\sum_{j=1}^{r}k_{\varphi_j}+\sum_{j=1}^{r}m_{\varphi_j} \pmod2 $$
by induction hypothesis, where we have used Lemma~\ref{lem:Q} part $(iii)$, then rearranged the terms, and finally exploited Lemma~\ref{lem:n_f} at the last transition. We claim that $(k_{\varphi_i}+k_{\varphi_{i+1}})+(m_{\varphi_i}+m_{\varphi_{i+1}}) \equiv 0 \pmod2$ for $i=1,3, \ldots ,r-1$.

We will show that for every permutational polynomial $g(x)=b_0+b_1x+b_2x^2+\cdots +b_tx^t$ satisfying the two conditions $n_g \equiv 0 \pmod2$ and $b_2 \equiv 0 \pmod2$, we have $(k_{\ggz}+k_{\ggu})+(m_{\ggz}+m_{\ggu}) \equiv 0 \pmod2$. This will be enough to prove the claim.

First we have $a_1^{\gz}+a_1^{\gu}=2k_{\ggz}+1+2k_{\ggu}+1$, thus from equations~\eqref{equ:sec0} and \eqref{equ:sec1} we have $$k_{\ggz}+k_{\ggu}=\frac{a_1^{\gz}+a_1^{\gu}}{2}-1=\frac{b_1+(b_1+2b_2+ \cdots +tb_t)}{2}-1=$$$$b_1+b_2+(2b_4+3b_6+4b_8+ \cdots )+\frac{3b_3+5b_5+ \cdots }{2}-1=$$$$2k_g+1+b_2+(2b_4+4b_8+ \cdots )+(3b_6+5b_{10}+ \cdots )+\frac{3b_3+5b_5+ \cdots }{2}-1$$
Using Lemma~\ref{lem:odd}, we can write the sum modulo $2$ as $$k_{\ggz}+k_{\ggu} \equiv (b_6+b_{10}+ \cdots )+\frac{3b_3+5b_5+ \cdots }{2} \pmod2$$
Again from equations~\eqref{equ:sec0} and \eqref{equ:sec1}, we can write $m_{\ggz}+m_{\ggu}$ modulo $2$ as
$$m_{\ggz}+m_{\ggu} \equiv b_2+(b_2+3 \cdot 1b_3+2 \cdot 3b_4+5 \cdot 2b_5+3 \cdot 5b_6+ \cdots )$$$$ \equiv 2b_2+(2 \cdot 3b_4+5 \cdot 2b_5+ \cdots )+(3 \cdot 1b_3+3 \cdot 5b_6+ \cdots )$$$$ \equiv b_3+b_6+b_7+b_{10}+b_{11}+b_{14}+b_{15}+ \cdots \pmod2$$
The last equivalence comes from applying Lemma~\ref{lem:odd}. Therefore the sum
$$(k_{\ggz}+k_{\ggu})+(m_{\ggz}+m_{\ggu})$$$$ \equiv (b_3+b_7+b_{11}+b_{15}+ \cdots)+\frac{3b_3+5b_5+ \cdots }{2}$$$$ \equiv \frac{5(b_3+b_5)+9(b_7+b_9)+ \cdots }{2} \pmod2$$
From Lemma~\ref{lem:mod4}, we have $5(b_3+b_5)+9(b_7+b_9)+ \cdots \equiv (b_3+b_5)+(b_7+b_9)+ \cdots \equiv 2n \pmod4$ and thus $\displaystyle{\frac{5(b_3+b_5)+9(b_7+b_9)+ \cdots }{2}} \equiv n \equiv 0 \pmod2$. The proof is now complete.
\end{proof}
\end{theorem}
If a mapping $f$ acting on the rooted binary tree $\{0,1\}^*$ is level transitive, then the orbit of every element in the boundary of the tree $\Z_2$ under $f$ is dense in $\Z_2$. Thus the dynamical system $(\Z_2,f)$ is minimal. Also Anashin proved in~\cite{ana:erg} that a polynomial $f \in \Z[x]$ is ergodic with respect to the normalized Haar measure on $\Z_2$ if and only if $f$ is transitive $\mathrm{mod}\ 2^n$ for every positive integer $n$. Combining Theorem~\ref{the:riv} and Theorem~\ref{the:main}, we thus obtain a new and more elementary proof of the result of Larin~\cite{larin:tr}.

\begin{theorem}[\cite{larin:tr}]
\label{the:larin}
Let $f(x)=a_0+a_1x+ \cdots + a_tx^t \in \Z_2[x]$. Then the dynamical system $(\Z_2,f)$ is minimal, or equivalently, $f$ is ergodic with respect to the normalized Haar measure of $\Z_2$ if and only if the following conditions are satisfied:
\begin{itemize}
\item[(i)] $a_0 \equiv 1 \pmod2$
\item[(ii)] $a_1 \equiv 1 \pmod2$
\item[(iii)] $2a_2 \equiv a_3+a_5+ \cdots \pmod4$
\item[(iv)] $a_2+a_1-1 \equiv a_4+a_6+ \cdots \pmod4$
\end{itemize}
\end{theorem}

\bibliographystyle{alpha}

\newcommand{\etalchar}[1]{$^{#1}$}
\def\cprime{$'$}

\end{document}